\documentclass[12pt]{amsart}
\usepackage{amsfonts,amssymb}
\usepackage{amsmath}
\usepackage{amsthm,amscd,latexsym}
\usepackage{mathrsfs}

\newtheorem{theorem}{Theorem}[section]

\newtheorem{Definition}[theorem]{Definition}
\newtheorem{lemma}[theorem]{Lemma}
\newtheorem{proposition}[theorem]{Proposition}
\newtheorem{Example}[theorem]{Example}
\newtheorem{Remark}[theorem]{Remark}

\newenvironment{remark}{\begin{Remark}\begin{em}}{\end{em}\end{Remark}}
\newenvironment{example}{\begin{Example}\begin{em}}{\end{em}\end{Example}}
\newenvironment{definition}{\begin{Definition}\begin{em}}{\end{em}\end{Definition}}

\makeatletter

\@addtoreset{equation}{section}
\makeatother

\newcommand{\R}{{\mathbb R}}

\DeclareMathOperator*{\argmin}{\mathrm{arg\, min}}

\DeclareMathOperator{\diam}{\mathrm{diam}}
\DeclareMathOperator{\var}{\mathrm{var}}
\DeclareMathOperator{\supp}{\mathrm{supp}}

\begin{document}

\title[Discrete-time gradient flows in Alexandrov spaces]{Discrete-time gradient flows and law of large numbers in Alexandrov spaces}
\author[Shin-ichi Ohta and Mikl\'os P\'alfia] {Shin-ichi Ohta and Mikl\'os P\'alfia}
\address{Department of Mathematics, Kyoto University, Kyoto 606-8502, Japan}
\email{sohta@math.kyoto-u.ac.jp}
\email{palfia.miklos@aut.bme.hu}

\thanks{\emph{AMS Classifications} (2010): 51K05, 53C20, 58C05, 49L20, 49M37\\
S.O. is supported by the Grant-in-Aid for Young Scientists (B) 23740048;
M.P. is supported by the Research Fellowship of the Canon Foundation.}
 \keywords{Alexandrov space, convex function, discrete time gradient flow, law of large numbers}

\begin{abstract}
We develop the theory of discrete-time gradient flows for convex functions on Alexandrov spaces
with arbitrary upper or lower curvature bounds.
We employ different resolvent maps in the upper and lower curvature bound cases to construct such a flow,
and show its convergence to a minimizer of the potential function.
We also prove a stochastic version, a generalized law of large numbers for convex function valued random variables,
which not only extends Sturm's law of large numbers on nonpositively curved spaces to arbitrary lower or upper 
curvature bounds, but this version seems new even in the Euclidean setting.
These results generalize those in nonpositively curved spaces (partly for squared distance functions)
due to Ba\v{c}\'ak, Jost, Sturm and others,
and the lower curvature bound case seems entirely new.
\end{abstract}

\maketitle

\section{Introduction}

In this paper, we consider discrete-time gradient flows for convex functions on \emph{Alexandrov spaces} $(X,d)$
with arbitrary upper or lower curvature bounds.
An Alexandrov space is a metric space whose sectional curvature is bounded above or below
by some constant in the sense of triangle comparison theorem (see Section~\ref{S:Alex}).
The \emph{discrete-time gradient flow} is introduced with an appropriate notion of \emph{resolvent operator}
$J_{\lambda}^f:X\to X$ defined for a fixed geodesically convex function $f$ and a positive number $\lambda>0$.
The operator $J_{\lambda}^f$ provides a gradient descent step towards the set of minimizers of $f$.
Under upper and lower curvature bounds, we define our $J_{\lambda}^f$ differently.
In the case of upper curvature bounds, we employ the standard \emph{Moreau--Yosida resolvent}:
\begin{equation}\label{eq:J-up}
J_{\lambda}^f(x):=\argmin_{y\in G} \left\{ f(y)+\frac{1}{2\lambda}d(x,y)^2 \right\}
\end{equation}
for a closed geodesically convex set $G$ containing a nonempty sublevel set of $f$.
In the case of lower curvature bounds, we define
\begin{equation}\label{eq:J-low}
J_{\lambda}^f(x):=g\textrm{-\!}\exp(\lambda\nabla(-f)(x)),
\end{equation}
where $g\textrm{-\!}\exp$ is the gradient exponential map and $\nabla(-f)(x)$ denotes the gradient vector of $-f$
(see \cite{PP,Pet07} and Sections~\ref{S:convex}, \ref{S:J} for these notions).
Before discussing the reason why we use these different resolvents, we present our results in this paper.

With these mappings at hand, we define the sequence
\[ x_{k+1}:=J_{\lambda_k}^f(x_k) \]
for $k \ge 0$ with an arbitrary starting point $x_0$ and for an a priori given positive sequence $\lambda_k>0$.
We prove the convergence of $x_k$ to a minimizer of $f$ under various, plausible conditions on $f$ and the sequence $\lambda_k$.
In particular, we generalize the classical results in \cite{Brezis} to arbitrary Alexandrov spaces.
Furthermore, our results generalize the ones recently given in \cite{Ba1} for \emph{NPC spaces}
(Alexandrov spaces with upper curvature bound by $0$) to arbitrary Alexandrov spaces.
In the upper curvature bound case, we allow $X$ to be infinite dimensional,
while in the lower curvature bound case we formulate our results for finite dimensions for technical reasons,
although our techniques would work in infinite dimensions equally well.
The most general known results in the literature, according to our knowledge,
consider NPC spaces and Riemannian manifolds with nonpositive sectional curvature,
see for example \cite{Ben12,Fer02,Li} just to mention a few among the numerous results. Also our results relate to and generalize the ones given in \cite{Jost-Calc,Jost,Jost-NPC,Jost-NLD} for NPC spaces, and harmonic maps with NPC target spaces.

We also consider the case of $f(x)=\sum_{i=1}^nf_i(x)$, where $f_i$ are also geodesically convex functions.
Then, under the assumption of $\sum_{k=0}^\infty\lambda_k^2<+\infty$ and the Lipschitz continuity of $f_i$,
we prove that the sequence generated by
\begin{equation}\label{eq:intr1}
x_{k+1}:=J_{\lambda_k}^{f_1}\circ\cdots\circ J_{\lambda_k}^{f_n}(x_k)
\end{equation}
converges to a minimizer of $f$ in any Alexandrov space.
On the one hand, this result generalizes the ones given for Euclidean spaces in \cite{Bert1,Bert2,Ne1,Ne2} and for NPC spaces in \cite{Ba2}.
On the other hand, this is also a generalization of the ``no dice'' approximation result given in NPC spaces for the \emph{barycenter},
which is the minimizer of $f(x)=\sum_{i=1}^n w_i d(x,a_i)^2$ with fixed points $a_i \in X$, in \cite{LP2,Hol}.
The barycenter (sometimes also called the \emph{Karcher mean} indebted to \cite{Ka77}),
or more generally the \emph{$p$-mean} obtained as the unique minimizer of
$f(x)=\sum_{i=1}^nw_id(x,a_i)^p$ for $p\in [1,+\infty)$, is of great interest, see for example \cite{Arn05,Arn10,Bha,BH,Jost-NPC,LL1,Ken90,Ken91}.
Our general approximation results, motivated by and applied for $p$-means among many others,
carry over to positively curved Riemannian setting, for example, compact Lie groups with bi-invariant Riemannian metrics
\cite{Arn05,Arn10,Moakher,Ken90,Ken91}, see Remarks~\ref{R:2} and \ref{R:3}.

We also prove a stochastic version of the convergence of the discrete-time flow given in \eqref{eq:intr1}.
In this setting, we assume that
\[ f(x):=\int_{F_K(G)}h(x)d\mu(h), \]
where $\mu$ is a probability measure supported over the cone of lower semi-continuous,
$K$-convex functions $F_K(G)$ over $G$, with $K>0$.
Then we prove a \emph{law of large numbers} result for the stochastic sequence
\[ x_{k+1}:=J_{\lambda_k}^{f_k}(x_k), \]
where $f_k$ is a sequence of independent, identically distributed $F_K(G)$-valued random variables with distribution $\mu$.
That is to say, we prove that $x_k\to \mathbb{E}\mu$ almost surely,
where $\mathbb{E}\mu$ is the (unique) minimizer of $f(x)=\int_{F_K(G)}h(x)d\mu(h)$.
This generalizes a result of Sturm \cite{St02b,St03},
which states that $x_k\to \mathbb{E}\mu$ almost surely
when $\mu$ is supported only on squared distance functions $g_a(x):=d(a,x)^2$ on NPC spaces.
This result of Sturm already implies the classical \emph{law of large numbers} on Euclidean spaces,
since on them $\mathbb{E}\mu=\int_{G}ad\tilde{\mu}(a)$,
where $\tilde{\mu}$ is the push-forward measure of $\mu$ under the map $g_a$.
Hence our result extends the law of large numbers to arbitrary Alexandrov spaces and arbitrary convex functions,
see Remarks~\ref{R:2} and \ref{R:3}.
Sturm \cite{St01,St02a,St02b,St05} used his result in his stochastic approach to the theory of harmonic maps between metric spaces,
and also his result became extremely useful for the barycenter in the case of the NPC space of positive definite matrices \cite{LL1,LP2,Hol}.
Therefore we expect wide applicability of our results, for example in the case of positive curvature.

The definitions and properties of $J_{\lambda}^f$ distinguish two different kinds of approaches
in the lower and upper curvature bound cases.
In \cite{ags,Jost-Calc,Jost-NPC,Jost-NLD} among many others,
for setting up the minimizing movements, the original resolvent \eqref{eq:J-up} given in \cite{Brezis}
is being used that we also adopt in the upper curvature bound case.
Besides technical reasons, the usefulness of \eqref{eq:J-up} in discrete-time gradient flows is due to the fact that
Alexandrov spaces with upper curvature bounds are simply connected and have unique minimal geodesics
in balls with designated radii.
We cannot expect these properties in the lower curvature bound case, the injectivity radius can be $0$ even locally.
Then it is difficult to control the behavior of discrete-time flows and there are no investigation in this direction as far as the authors know,
while continuous-time gradient flows are intensively studied in \cite{PP,Pet07,Ly,Oh1,Jost-NPC}.

To overcome this difficulty, we introduce the other (but natural) construction \eqref{eq:J-low}
relying on gradient vectors directly.
This makes an interesting contrast with \eqref{eq:J-up}:
\begin{align}
\log_{x_{k+1}}x_k &=\lambda\nabla f(x_{k+1}), \label{eq:gf-up}\\
\log_{x_{k}}x_{k+1} &=\lambda\nabla(-f)(x_{k}) \label{eq:gf-low},
\end{align}
in the upper and lower curvature bounds, respectively,
where $\log_x y$ is the direction from $x$ to $y$.
In other words, these two flows provided in the opposite curvature bounds are in reverse relation.
In the upper curvature bound case, we take the backward flow for the convex function $f$,
while we take the forward flow for the concave function $-f$ in the lower curvature bound case.
In Euclidean spaces, both methods work equivalently well \cite{Bert1,Bert2,Ne1,Ne2}.
In general, it seems that the curvature bound determines whether a proximal step \eqref{eq:J-up}
or a gradient step \eqref{eq:J-low} is more suitable from the analytic point of view of discrete flows.
For instance, the convexity of squared distance functions, which is the very definition of upper curvature bounds,
can give a contraction estimate of discrete-time gradient flows together with \eqref{eq:gf-up}
(estimate $d(x_{k+1},y_{k+1})$ from above by using $d(x_k,y_k)$
and the convexity of $f$ along a geodesic between $x_{k+1}$ and $y_{k+1}$).
Similarly, the concavity of squared distance functions is useful only with \eqref{eq:gf-low}
(via the convexity of $f$ along a geodesic between $x_k$ and $y_k$).

\section{Alexandrov spaces}\label{S:Alex}

We refer to \cite{bbi} for the basics of metric geometry and Alexandrov spaces.
Let $(X,d)$ be a metric space.
A continuous curve $\gamma:[0,1] \to X$ is called a \emph{minimal geodesic}
if it satisfies $d(\gamma(s),\gamma(t))=|s-t|d(\gamma(0),\gamma(1))$ for all $s,t \in [0,1]$.
We say that $(X,d)$ is \emph{geodesic} if any two points $x,y \in X$ admit a minimal geodesic between them.
Though minimal geodesics are not necessarily unique, we abuse the notation $x\#_ty$, $t \in [0,1]$,
for denoting a minimal geodesic from $x$ to $y$.
A subset $G \subset X$ is said to be \emph{geodesically convex} if, for any $x,y \in G$,
all minimal geodesics $x\#_ty$ between them are contained in $G$.

For $\kappa\in\mathbb{R}$, we denote by $\mathbb{M}^2(\kappa)$ a complete, simply connected,
$2$-dimensional Riemannian manifold of constant sectional curvature $\kappa$.
For three points $x,y,z\in X$ with $d(x,y)+d(y,z)+d(z,x)<2\pi/\sqrt{\kappa}$ if $\kappa>0$,
we can take corresponding points $\tilde{x},\tilde{y},\tilde{z}\in\mathbb{M}^2(\kappa)$
uniquely up to rigid motions such that
\[ d_{\mathbb{M}^2(\kappa)}(\tilde{x},\tilde{y})=d(x,y),\quad
d_{\mathbb{M}^2(\kappa)}(\tilde{y},\tilde{z})=d(y,z),\quad
d_{\mathbb{M}^2(\kappa)}(\tilde{z},\tilde{x})=d(z,x). \]
We call $\triangle\tilde{x}\tilde{y}\tilde{z}$ a \emph{comparison triangle} of $\triangle xyz$ in $\mathbb{M}^2(\kappa)$.

\begin{definition}[Alexandrov spaces]\label{curvaturebounds}
A geodesic metric space $(X,d)$ is called an \emph{Alexandrov space of curvature bounded above by $\kappa$} if,
for any $x,y,z\in X$ with $d(x,y)+d(y,z)+d(z,x)<2\pi/\sqrt{\kappa}$ if $\kappa>0$, we have
\begin{equation}\label{eq:up}
d(y\#_tz,x)\leq d_{\mathbb{M}^2(\kappa)}(\tilde{y}\#_t\tilde{z},\tilde{x})
\end{equation}
for any minimal geodesic $y\#_tz$ joining $y$ and $z$.

Similarly, $(X,d)$ is called an \emph{Alexandrov space of curvature bounded below by $\kappa$} if we have
\begin{equation}\label{eq:low}
d(y\#_tz,x)\geq d_{\mathbb{M}^2(\kappa)}(\tilde{y}\#_t\tilde{z},\tilde{x})
\end{equation}
for any minimal geodesic $y\#_tz$.
\end{definition}

For instance, if $\kappa=0$, then \eqref{eq:up} is calculated as
\[ d(x,y\#_tz)^2\leq (1-t)d(x,y)^2+td(x,z)^2-t(1-t)d(y,z)^2, \]
and \eqref{eq:low} is
\[ d(x,y\#_tz)^2\geq (1-t)d(x,y)^2+td(x,z)^2-t(1-t)d(y,z)^2. \]
By the parallelogram identity, Hilbert spaces have curvature bounded both above and below by $0$.
Here are some further examples.

\begin{example}\label{Ex:Alex}
(1)
A complete, simply connected Riemannian manifold with the Riemannian distance
is an Alexandrov space with curvature bounded above by $\kappa$ if and only if
its sectional curvature is not greater than $\kappa$.
Typical examples of nonpositively curved spaces admitting singularities include
trees, Euclidean buildings and gluing of nonpositively curved spaces.
See \cite[\S 9.1]{bbi} for further examples.

(2)
A complete Riemannian manifold is an Alexandrov space of curvature bounded below by $\kappa$
if and only if its sectional curvature is not less than $\kappa$.
Typical examples of nonnegatively curved spaces admitting singularities include
the boundaries of convex domains in Euclidean spaces,
quotients of nonnegatively curved spaces by isometries (e.g., orbifolds),
and the $L^2$-Wasserstein spaces over nonnegatively curved spaces (see \cite{St061, Oh1}).
We refer to \cite[\S 10.2]{bbi} for further examples.
\end{example}

An important feature of Alexandrov spaces is that angles are well defined between two geodesics
$\gamma$ and $\eta$ emanating from the same point $\gamma(0)=\eta(0)=x$:
\[ \angle_x(\gamma,\eta):=\lim_{s,t\to0+}\angle\widetilde{\gamma(t)}\tilde{x}\widetilde{\eta(s)}, \]
where $\widetilde{\gamma(t)}\tilde{x}\widetilde{\eta(s)}$ is a comparison triangle in $\mathbb{M}^2(\kappa)$.
For fixed $x\in X$, we define $\Sigma_x'X$ as the set of unit speed minimal geodesics $\gamma:[0,\delta]\to X$,
$\delta>0$, emanating from $x$.
The angle $\angle_x(\gamma,\eta)$ defines a pseudo-distance on the set $\Sigma_x'X$.
The completion of $(\Sigma_x'X/\{\angle_x=0\},\angle_x)$ with respect to $\angle_x$
is denoted by $(\Sigma_xX,\angle_x)$, and is called the \emph{space of directions} at $x\in X$.

The \emph{tangent cone} $(C_xX,\sigma_x)$ at $x\in X$ is defined as the Euclidean cone over $(\Sigma_xX,\angle_x)$:
\[ C_xX:=(\Sigma_xX\times[0,\infty))/\sim, \]
where $(\gamma,0)\sim(\eta,0)$ and
\[ \sigma_x((\gamma,s),(\eta,t)):=\sqrt{s^2+t^2-2st\cos\angle_x(\gamma,\eta)} \]
for $(\gamma,s),(\eta,t)\in C_xX$.
We denote by $o_x$ the origin $(*,0)\in C_xX$.
For any $u=(\gamma,s),v=(\eta,t)\in C_xX$, we can define their inner product as
\[ \left\langle u,v\right\rangle:=st\cos\angle_x(\gamma,\eta). \]

If $X$ is complete, finite Hausdorff dimensional and has curvature bounded above in the sense of Definition~\ref{curvaturebounds},
then $(\Sigma_xX,\angle_x)$ is an Alexandrov space of curvature bounded above by $1$
and $(C_xX,\sigma_x)$ is an Alexandrov space of curvature bounded above by $0$.
In the case of curvature bounded below, we have the same curvature bounds
for $(\Sigma_xX,\angle_x)$ and $(C_xX,\sigma_x)$, but from below.
In the infinite dimensional case, however, this is not the case in general.

By the definition of the angle, we readily have the following (see \cite[Corollary~4.5.7]{bbi}).

\begin{theorem}[First variation formula]\label{T:1vf}
Let $\gamma:[0,\delta] \to X$ be a geodesic in an Alexandrov space $(X,d)$ with curvature bounded above or below by $\kappa$,
and assume that $X$ is locally compact in the lower curvature bound case.
Put $x:=\gamma(0)$ and take $y \in X$ with $d(x,y)<\pi/\sqrt{\kappa}$ if $\kappa>0$.
Then the function $d(t):=d(\gamma(t),y)$ satisfies
\begin{equation}\label{eq:1vf}
\lim_{\epsilon\to 0+}\frac{d(\epsilon)-d(0)}{\epsilon}
 =-\frac{1}{d(x,y)} \max_{\eta} \left\langle \gamma'(0),\eta'(0) \right\rangle,
\end{equation}
where $\eta:[0,1] \to X$ runs over all minimal geodesics from $x$ to $y$.
\end{theorem}

We remark that $\eta$ is unique under the upper curvature bound.
The inequality `$\le$' holds in \eqref{eq:1vf} for any $\eta$ in a more general situation
without any compactness assumption (\cite[Proposition~4.5.2]{bbi}),
and such an inequality is enough in our discussion in the lower curvature bound case.
Equality in \eqref{eq:1vf} is necessary only in the proof of Lemma~\ref{L:1}(I).

\section{Convex functions on Alexandrov spaces}\label{S:convex}

Let $(X,d)$ be an Alexandrov space with curvature bounded above or below by $\kappa\in\mathbb{R}$.
We say that a function $f:X\to (-\infty,\infty]$ is \emph{$K$-$($geodesically$)$ convex} for $K \in \mathbb{R}$ if
\begin{equation}\label{k-convexity}
f(x\#_ty)\leq (1-t)f(x)+tf(y)-\frac{K}{2}t(1-t)d(x,y)^2
\end{equation}
holds for any $x,y\in X$, $t\in[0,1]$ and any minimal geodesic $x\#_t y$.
The $0$-convexity will be simply called the convexity.

\begin{definition}[Absolute gradients]
Let $f:X\to(-\infty,\infty]$ be lower semi-continuous and $K$-convex.
Then the (descending) \emph{absolute gradient} of $f$ at $x \in X$ with $f(x) \neq \infty$ is defined by
\[ |\nabla\!_- f|(x):=\max\left\{0,\limsup_{y\to x}\frac{f(x)-f(y)}{d(x,y)}\right\}. \]
Note that $|\nabla\!_- f|(x)\in[0,\infty]$ and also $|\nabla\!_- f|(x)=0$ if $f(x)=\inf_{y\in X}f(y)$.
\end{definition}

\begin{definition}[Directional derivatives]
For $f:X\to(-\infty,\infty]$, the \emph{directional derivative} of $f$ at $x$ with $f(x) \neq \infty$
in the direction $v \in C_xX$ is defined as
\[ D_xf(v):=\liminf_{(\gamma,s) \to v} \left\{ \lim_{t\to 0+}\frac{f(\gamma(st))-f(x)}{t} \right\}, \]
where $(\gamma,s)\in \Sigma'_xX \times [0,\infty) \subset C_xX$.
\end{definition}

The above limit along $\gamma$ indeed exists for lower semi-continuous, $K$-convex functions.
Note that $D_xf(v) \ge -|\nabla\!_- f|(x) \cdot |v|$ clearly holds.
Typical examples of $K$-convex (or $K$-concave) functions are squared distance functions.
We set $d_a(x):=d(a,x)$ for $a,x \in X$, and denote closed metric balls by
\[ \bar{B}_a(r):=\{x\in X:d(a,x)\le r\},\quad a \in X,\ r>0. \]

\begin{proposition}[Proposition~3.1 in \cite{Oh2}]\label{P:Convex}
Let $(X,d)$ be an Alexandrov space with curvature bounded above by $\kappa>0$.
Then, for any $a \in X$, the function $d_a^2$ is $K$-convex on the $($geodesically convex$)$ metric ball
$\bar{B}_a(r)$ with $2r=(\pi/2-\epsilon)/\sqrt{\kappa}$ and $K=(\pi-2\epsilon)\tan\epsilon$
for arbitrary $\epsilon\in(0,\pi/2)$.
\end{proposition}

\begin{proposition}[Lemma 3.3 in \cite{Oh1}]\label{P:Convex2}
Let $(X,d)$ be an Alexandrov space with curvature bounded below by $\kappa<0$.
Then, for any $a \in X$, the function $-d_a^2$ is $K$-convex on the metric ball
$\bar{B}_a(r)$ with $K=-2(1-\kappa(2r)^2)$ for all $r>0$.
\end{proposition}

Slightly more generally, we can take a geodesically convex set $G \subset X$
with $\diam G \le 2r$ in Propositions~\ref{P:Convex}, \ref{P:Convex2}.

In the lower curvature bound case, by comparing the convexity of $f$ and the concavity of the squared distance function,
one can find the useful notion of gradient vectors as follows (see \cite{PP, Pet07, Ly, Oh1} for details).
Let $(X,d)$ be an Alexandrov space of curvature bounded below
and $f:X \to (-\infty,\infty]$ be a lower semi-continuous, $K$-convex function.
Then, at every $x \in X$ with $0<|\nabla\!_- f|(x)<\infty$,
we can find the unique direction $\gamma \in \Sigma_x X$ such that $D_x f(\gamma)=-|\nabla\!_- f|(x)$ and
\begin{equation}\label{eq:grad}
D_{x}f(\eta) \geq -|\nabla\!_- f|(x) \langle \gamma,\eta \rangle
\end{equation}
for all $\eta \in \Sigma_x X$.
Thus $\nabla(-f)(x):=(\gamma,|\nabla\!_- f|(x))\in C_xX$ can be regarded as the \emph{gradient vector} of $-f$ at $x$.
Set also $\nabla(-f)(x):=o_x$ if $|\nabla\!_- f|(x)=0$.
The gradient vector will be used to define an appropriate resolvent map for $f$.

\section{Resolvent maps}\label{S:J}

In this section, we introduce our key tool also appeared in \cite{Jost-Calc,Jost,Jost-NPC,Jost-NLD}, the \emph{resolvent map} $J^f_{\lambda}$,
to construct discrete-time gradient flows for convex functions.
We will adopt different definitions of $J^f_{\lambda}$ in the upper and lower curvature bound cases.
Throughout the section, let $f:X \to (-\infty,\infty]$ be a convex, lower semi-continuous function
not identically $+\infty$, and fix the size $\lambda >0$ of the discrete-time step
and a closed, geodesically convex set $G \subset X$ containing a nonempty sublevel set of $f$.

First, let $(X,d)$ be a complete Alexandrov space with curvature bounded above by $\kappa$.
If $\kappa>0$, then we assume $\diam G<\pi/(2\sqrt{\kappa})$.
In this case, we employ the standard resolvent map used in, e.g., \cite{Jost, Mayer, ags}.

\begin{definition}[Resolvent map, upper curvature bound case]
For each $x \in X$, we define
\begin{equation}\label{eq:resolventabove}
J_{\lambda}^f(x):=\argmin_{y\in G} \left\{ f(y)+\frac{1}{2\lambda}d(x,y)^2 \right\}.
\end{equation}
\end{definition}

\begin{lemma}
Let $(X,d)$, $G$ and $f$ be as above.
Then there exists a unique point $y \in G$ attaining the minimum \eqref{eq:resolventabove}.
\end{lemma}

\begin{proof}
By Proposition~\ref{P:Convex} and $\diam G<\pi/(2\sqrt{\kappa})$,
the function $y \mapsto f(y)+\frac{1}{2\lambda}d(x,y)^2$ is $K$-convex on $G$ for some $K>0$.
The rest of the argument can be obtained by a straightforward optimization argument,
see for example \cite[Lemma~2.4.8]{ags}.
\end{proof}

For any minimal geodesic $\gamma:[0,\delta] \to G$ with $\gamma(0)=J_{\lambda}^f(x)$,
we deduce from the first variation formula (Theorem~\ref{T:1vf}) that
\begin{equation}\label{eq:optimum}
D_{J_{\lambda}^f(x)}f(\dot{\gamma}(0))
 -\frac{1}{\lambda}\left\langle \dot{\eta}(0),\dot{\gamma}(0) \right\rangle \ge 0,
\end{equation}
where $\eta:[0,1] \to X$ is any minimal geodesic from $J_{\lambda}^f(x)$ to $x$.

Next we consider the lower curvature bound case.
In this setting, the definition provided by \eqref{eq:resolventabove} for $J_{\lambda}^f(x)$ is not convenient,
because the squared distance function is no longer convex, but is concave instead.
This concavity leads to, however, the advantage of well defined gradient vectors of $-f$.
Then we shall define the resolvent map by using an ``exponential map'' from $C_x X$ to $X$.
Although we can not simply use geodesics since there may be no geodesic with a given initial direction,
the gradient curves of the convex function $-d_x^2$ will do the job.

Let $(X,d)$ be a complete, finite dimensional Alexandrov space of curvature bounded below by $\kappa$
with $\partial X=\emptyset$.
Note that, even for $\kappa<0$, the function $-d_x^2$ is $K$-convex on balls $\bar{B}_x(r)$
for some $K=K(\kappa,r)<0$ by Proposition~\ref{P:Convex2}.
Hence we can construct the \emph{gradient flow} $\Phi:[0,\infty) \times X \to X$ of $-d_x^2$,
i.e., each curve $\xi(t)=\Phi(t,y)$ satisfies $\dot{\xi}(t)=\nabla(d_x^2)(\xi(t))$ at almost all $t>0$.
The convexity of $-d_x^2$ ensures the uniqueness and contraction of $\Phi$, see \cite{PP,Pet07}.
The \emph{gradient exponential map} $g\textrm{-\!}\exp_x:C_xX \to X$ is obtained by a re-parametrized scaling of $\Phi$:
Define $g\textrm{-\!}\exp_x$ as the limit of the map
\[ \Phi(s,\cdot) \circ \mathrm{id}_X:(X,e^s d) \to (X,d) \]
as $s \to \infty$, where $e^s d$ is the scaled distance and $\mathrm{id}_X:(X,e^s d) \to (X,d)$ is the identity map.
The gradient exponential map enjoys many nice properties, for instance,
the curve $\xi(t)=g\textrm{-\!}\exp_x(tv)$ satisfies
\begin{equation}\label{eq:gexp}
\dot{\xi}(0)=v,\quad \dot{\xi}(t)=\frac{d(x,\xi(t))}{t} \nabla d_x(\xi(t)), \quad d(x,\xi(t)) \le t|v|.
\end{equation}
Moreover, the following useful comparison estimate holds.

\begin{lemma}[Lemma~3.1.2 in \cite{Pet07}]\label{L:gexp}
Put $\xi(t)=g\textrm{-\!}\exp_x(tv)$ with $v \in C_x X$.
Then, for any $(-K)$-convex function $h:X \to \R$ with $K \ge 0$ and all $t>0$, we have
\[ h(\xi(t)) \ge h(x) +t D_x h(v) -\frac{K}{2}(t|v|)^2. \]
\end{lemma}

\begin{remark}\label{R:gexp}
The gradient flow $\Phi$ of $-d_x^2$ can be constructed also in proper, infinite dimensional Alexandrov spaces
(see \cite[Appendix]{PP} and \cite{Ly,Oh1}).
However, the proof of Lemma~\ref{L:gexp} above in \cite{Pet07} essentially requires
both $\dim X<\infty$ and $\partial X=\emptyset$.
One may consult the argument in \cite[Appendix]{PP} proving an estimate comparable to Lemma~\ref{L:gexp}
(called the \emph{monotonicity} there) along gradient curves parametrized in a different way.
However, these curves may be defined only on small intervals.
\end{remark}

We are ready to define the resolvent map under lower curvature bound.
We abuse the same notation $J_{\lambda}^f(x)$ as the upper curvature bound case.

\begin{definition}[Resolvent map, lower curvature bound case]\label{D:Jlower}
For $x \in X$ with $|\nabla\!_- f|(x)<\infty$, we define
\begin{equation}\label{eq:resolventbelow}
J_{\lambda}^f(x) := g\textrm{-\!}\exp(\lambda\nabla(-f)(x)),
\end{equation}
where $\nabla(-f) \in C_xX$ is the gradient vector of $-f$ at $x$ given in Section~\ref{S:convex}.
\end{definition}

The following estimates will play crucial roles in the next section.

\begin{lemma}\label{L:1}
Let $(X,d)$ be a complete Alexandrov space either with curvature bounded above or below by $\kappa$.
\begin{enumerate}
\item[(I)]
If $(X,d)$ has curvature bounded above by $\kappa>0$, then also assume $\diam G<\pi/(2\sqrt{\kappa})$.
Then we have
\begin{equation}\label{eq:L:1.1}
d(y,J_{\lambda}^f(x))^2\leq d(y,x)^2-2\lambda[f(J_{\lambda}^f(x))-f(y)]
\end{equation}
for all $x,y \in G$.

\item[(II)]
In the lower curvature bound case, we assume that $X$ is finite dimensional, $\partial X=\emptyset$,
and that $\diam G<\infty$ if $\kappa<0$.
Then we have
\begin{equation}\label{eq:L:1.2}
d(y,J_{\lambda}^f(x))^2\leq d(y,x)^2-2\lambda[f(x)-f(y)]+\frac{K}{2}(\lambda |\nabla\!_- f|(x))^2
\end{equation}
for all $x,y \in G$ satisfying $J_{\lambda}^f(x) \in G$, where $K=K(\kappa,\diam G) \ge 0$.
\end{enumerate}
\end{lemma}

\begin{proof}
(I) By assumption, the squared distance function is convex (Proposition~\ref{P:Convex}).
Hence, by Theorem~\ref{T:1vf},
\[ d(y,x)^2 \ge d(y,J_{\lambda}^f(x))^2
 -2\left\langle \dot{\gamma}(0),\dot{\eta}(0) \right\rangle \]
for $\gamma(t)=J_{\lambda}^f(x) \#_t y$ and $\eta(t)=J_{\lambda}^f(x) \#_t x$
(they are unique since $\diam G<\pi/(2\sqrt{\kappa})$).
Combine this with \eqref{eq:optimum} to get
\begin{equation}\label{eq:eq1}
d(y,J_{\lambda}^f(x))^2\leq d(y,x)^2+2\lambda D_{J_{\lambda}^f(x)}f(\dot{\gamma}(0)).
\end{equation}
Now the convexity of $f$ along $\gamma$ yields that $D_{J_{\lambda}^f(x)}f(\dot{\gamma}(0))\leq f(y)-f(J_{\lambda}^f(x))$,
so we get from \eqref{eq:eq1} that
\begin{equation*}
d(y,J_{\lambda}^f(x))^2\leq d(y,x)^2-2\lambda[f(J_{\lambda}^f(x))-f(y)].
\end{equation*}

(II) Put $\xi(\lambda):=g\textrm{-\!}\exp(\lambda\nabla(-f)(x))$.
By Proposition~\ref{P:Convex2}, the function $-d_y^2$ is $(-K)$-convex on $G$ for some $K=K(\kappa,\diam G) \ge 0$.
Thus Lemma~\ref{L:gexp} shows that
\[ d(y,\xi(\lambda))^2 \le d(y,x)^2 +\lambda D_x(d_y^2)(\nabla(- f)(x))
 +\frac{K}{2}(\lambda |\nabla\!_- f|(x))^2. \]
Fixing arbitrary minimal geodesic $\gamma(t)=x\#_t y$,
we deduce from the first variation formula (Theorem~\ref{T:1vf}) and \eqref{eq:grad} that
\[ D_x(d_y^2)(\nabla(-f)(x)) \le -2\langle \dot{\gamma}(0),\nabla(-f)(x) \rangle
 \le 2D_x f(\dot{\gamma}(0)). \]
Finally the convexity of $f$ shows that $f(y) \ge f(x) +D_x f(\dot{\gamma}(0))$.
Therefore we obtain
\[ d(y,J_{\lambda}^f(x))^2\leq d(y,x)^2-2\lambda[f(x)-f(y)]+\frac{K}{2}(\lambda |\nabla\!_- f|(x))^2. \]
\end{proof}

\section{Proximal and sub-gradient methods}\label{S:nodice}

The resolvent map $J_{\lambda}^f$ can be used to consider proximal point algorithms or,
in other words, discrete-time gradient flows for general convex functions in the upper curvature bound case.
We start with a basic result that generalizes the one in \cite{Ba1} given in NPC spaces.
The algorithm has been used at many places, one of the first occasions was in \cite{Brezis}.
The situation is the same as Lemma~\ref{L:1}(I).

\begin{theorem}\label{T:3}
Let $(X,d)$ be a complete Alexandrov space with curvature bounded above by $\kappa$.
Let $f:X\to (-\infty,\infty]$ be a convex, lower semi-continuous function
and $G \subset X$ be a closed, geodesically convex set containing a sublevel set of $f$
such that $\diam G<\pi/(2\sqrt{\kappa})$ if $\kappa>0$.
Take a positive sequence $\{\lambda_k\}_{k \ge 1}$ with $\sum_{k=1}^\infty\lambda_k=+\infty$.
Fix an arbitrary starting point $x_0\in G$ and put
\[ x_{k}:=J_{\lambda_k}^f(x_{k-1}), \qquad k \ge 1. \]
Then we have $\lim_{k \to \infty}f(x_k)=\inf_{y \in G}f(y)$.
\end{theorem}

\begin{proof}
By the definition \eqref{eq:resolventabove} of $J_{\lambda}^f$, the sequence $f(x_k)$ is monotone non-increasing.
Indeed, we have
\[ f(J_{\lambda_{k+1}}^f(x_k))+\frac{1}{2\lambda_{k+1}}d(x_k,J_{\lambda_{k+1}}^f(x_k))^2\leq f(x_k). \]
Furthermore, by \eqref{eq:L:1.1} in Lemma~\ref{L:1}, we have for any $y \in G$
\[ d(y,x_{k+1})^2\leq d(y,x_k)^2-2\lambda_{k+1}[f(x_{k+1})-f(y)]. \]
This combined with the monotonicity of $f(x_k)$ yields
\[ 2[f(x_k)-f(y)]\sum_{i=1}^k\lambda_i\leq 2\sum_{i=1}^k\lambda_i[f(x_i)-f(y)]\leq d(y,x_0)^2-d(y,x_k)^2, \]
which gives
\[ f(x_k)-f(y)\leq \frac{d(y,x_0)^2}{2\sum_{i=1}^k\lambda_i}. \]
By the choice of $\lambda_k$, this implies that $\lim_{k \to \infty}f(x_k) \le f(y)$ for any $y \in G$.
Therefore we obtain $\lim_{k \to \infty}f(x_k)=\inf_{y \in G}f(y)$.
\end{proof}

\begin{remark}
Weak convergence in Alexandrov spaces with upper curvature bounds has been introduced in \cite{Esp09} which generalized this notion given for NPC spaces by Jost in \cite{Jost-Calc}.
The same results for weak convergence as in NPC spaces hold
if we restrict the analysis to closed metric balls of diameter at most $\pi/(2\sqrt{\kappa})$.
Hence actually one can prove weak convergence to a minimizer (if it exists) in Theorem~\ref{T:3}
in the same way as in \cite{Ba1} for NPC spaces.
Furthermore, if $f$ is $K$-convex with $K>0$, then we have (strong) convergence to the unique minimizer $y$.
\end{remark}

In the rest of this section, we set up a discrete-time gradient flow converging to a minimizer of a convex function
that is the sum of finitely many convex functions.
We adjust the setting of Lemma~\ref{L:1} to admit such sum of functions.

\begin{definition}[Proximal Point Algorithm]\label{D:alg1}
Let $(X,d)$ be a complete Alexandrov space either with curvature bounded above or below by $\kappa$,
and $G \subset X$ be a closed, geodesically convex set satisfying the following:
\begin{enumerate}
\item[(I)]
In the upper curvature bound case, $\diam G<\pi/(2\sqrt{\kappa})$ if $\kappa>0$;

\item[(II)]
In the lower curvature bound case, $\dim X<\infty$, $\partial X=\emptyset$,
and $\diam G<\infty$ if $\kappa<0$.
\end{enumerate}
Let $f_i:G\to (-\infty,\infty]$ be a convex, lower semi-continuous function for $i=1,\ldots,n$.
Set $f(x):=\sum_{i=1}^n f_i(x)$ and suppose that it is not identically $+\infty$.
Take a positive sequence $\lambda_k>0$ such that $\sum_{k=0}^{\infty}\lambda_k=+\infty$
and also $\sum_{k=0}^{\infty}\lambda_k^2<+\infty$.
Given $x_0 \in G$ and for each $k \ge 0$ and $1 \le i \le n$, we set
\[ x_{kn+i}:=J_{\lambda_k}^{f_i}(x_{kn+i-1}), \]
where the resolvent map is defined by \eqref{eq:resolventabove} or \eqref{eq:resolventbelow},
assuming that $x_m \in G$ for all $m \ge 0$ in the lower curvature bound case.
\end{definition}

Before turning to our result on the convergence of the sequences generated in Definition~\ref{D:alg1},
we state an elementary lemma from \cite[Lemma~3.4]{Bert1} for later use.

\begin{lemma}\label{L:2}
Let $a_k,b_k,c_k\geq 0$ be sequences such that $a_{k+1}\leq a_k-b_k+c_k$ for any $k \ge 1$,
and assume $\sum_{k=1}^{\infty}c_k<\infty$.
Then the sequence $a_k$ converges and also $\sum_{k=1}^{\infty}b_k<\infty$.
\end{lemma}

\begin{theorem}\label{T:1}
Let $(X,d)$, $G \subset X$, $f=\sum_{i=1}^n f_i$ and $\{\lambda_k\}_{k \ge 0}$ be as in Definition~$\ref{D:alg1}$.
Assume further that $X$ is locally compact, $f_i$ is $L$-Lipschitz for some $L \ge 1$ and all $i$,
and that $\inf_G f$ is attained at some point.
Then $x_m$ converges to some minimizer of $f$ in $G$ as $m \to \infty$.
\end{theorem}

\begin{proof}
Fix a minimizer $y \in G$ of $f$.

Upper curvature bound case (I):
By \eqref{eq:L:1.1} in Lemma~\ref{L:1}, we have
\[ d(y,x_{kn+i})^2\leq d(y,x_{kn+i-1})^2-2\lambda_k[f_i(x_{kn+i})-f_i(y)]. \]
Summing the above for $1\leq i\leq n$ implies
\[ d(y,x_{kn+n})^2\leq d(y,x_{kn})^2-2\lambda_k\sum_{i=1}^n[f_i(x_{kn+i})-f_i(y)], \]
which is equivalent to
\begin{equation}\label{eq:T:1.1}
d(y,x_{kn+n})^2 \leq
 d(y,x_{kn})^2-2\lambda_k\sum_{i=1}^n[f_i(x_{kn})-f_i(y)]+2\lambda_k\sum_{i=1}^n[f_i(x_{kn})-f_i(x_{kn+i})].
\end{equation}
The next step is to estimate $\sum_{i=1}^n[f_i(x_{kn})-f_i(x_{kn+i})]$ from above.
By \eqref{eq:resolventabove}, for any $1\leq j\leq n$, we have
\[ f_j(x_{kn+j})+\frac{1}{2\lambda_k}d(x_{kn+j},x_{kn+j-1})^2\leq f_j(x_{kn+j-1}), \]
which yields by using the $L$-Lipschitz continuity that
\[ d(x_{kn+j},x_{kn+j-1})\leq
 2\lambda_k\frac{f_j(x_{kn+j-1})-f_j(x_{kn+j})}{d(x_{kn+j},x_{kn+j-1})}\leq 2\lambda_k L. \]
Since $d(x_{kn},x_{kn+i})\leq \sum_{j=1}^{i}d(x_{kn+j-1},x_{kn+j})$, this gives also that
\begin{equation}\label{eq:T:1.2}
d(x_{kn},x_{kn+i})\leq 2\lambda_k Li.
\end{equation}
Furthermore, we have
\begin{align*}
\sum_{i=1}^n[f_i(x_{kn})-f_i(x_{kn+i})] &=\sum_{i=1}^n\sum_{j=1}^{i}[f_i(x_{kn+j-1})-f_i(x_{kn+j})]\\
&\leq \sum_{i=1}^n\sum_{j=1}^{i}Ld(x_{kn+j},x_{kn+j-1})\\
&\leq \sum_{i=1}^n2\lambda_k L^2i=\lambda_k L^2n(n+1).
\end{align*}
This combined with \eqref{eq:T:1.1} yields
\begin{equation}\label{eq:T:1.3}
d(y,x_{kn+n})^2\leq d(y,x_{kn})^2-2\lambda_k\sum_{i=1}^n[f_i(x_{kn})-f_i(y)]+2\lambda_k^2 L^2n(n+1).
\end{equation}
Since $f(x_{kn})-f(y) \ge 0$, Lemma~\ref{L:2} implies that the sequence $a_k:=d(y,x_{kn})^2$ converges and
\[ \sum_{k=0}^\infty\lambda_k[f(x_{kn})-f(y)]<+\infty. \]
Hence, by the assumption $\sum_{k=0}^\infty \lambda_k=+\infty$, there exists a subsequence $x_{k_ln}$ such that
$\lim_{l \to \infty}f(x_{k_ln})=f(y)$.
Since $x_{k_ln}$ is bounded, by local compactness it has a subsequence converging to a point $z \in G$,
which by lower semicontinuity of $f$ must be a minimizer of $f$.
Then, by replacing $y$ with $z$ in the above discussion,
the sequence $a_k=d(z,x_{kn})^2$ is convergent and has a subsequence converging to $0$.
Hence the whole sequence $a_k$ converges to $0$, i.e., $x_{kn}\to z$ as $k\to\infty$.
Moreover, \eqref{eq:T:1.2} gives
\[ d(z,x_{kn+i})\leq d(z,x_{kn})+d(x_{kn},x_{kn+i})\leq d(z,x_{kn})+2\lambda_k Li \]
for all $1\leq i\leq n$.
Since we have $\lambda_k\to 0$ by $\sum_{k=0}^{\infty}\lambda_k^2<+\infty$,
we conclude that $x_{kn+i}\to z$ as $k\to\infty$ for all $i$.
Therefore $x_m \to z$ as $m \to \infty$.

Lower curvature bound case (II):
The proof is similar to Case (I).
From \eqref{eq:L:1.2} in Lemma~\ref{L:1}, we get
\[ d(y,x_{kn+i})^2\leq d(y,x_{kn+i-1})^2-2\lambda_k[f_i(x_{kn+i-1})-f_i(y)]+\frac{K}{2}\lambda_k^2L^2. \]
Summing the above for $1\leq i\leq n$ yields
\[ d(y,x_{kn+n})^2\leq d(y,x_{kn})^2-2\lambda_k\sum_{i=1}^n[f_i(x_{kn+i-1})-f_i(y)]+\frac{K}{2}\lambda_k^2L^2 n, \]
which is equivalent to
\begin{equation}\label{eq:T:1.4}
\begin{split}
d(y,x_{kn+n})^2\leq &d(y,x_{kn})^2-2\lambda_k\sum_{i=1}^n[f_i(x_{kn})-f_i(y)]\\
 &+2\lambda_k\sum_{i=1}^n[f_i(x_{kn})-f_i(x_{kn+i-1})]
 +\frac{K}{2}\lambda_k^2L^2 n.
\end{split}
\end{equation}
We find by \eqref{eq:gexp} and assumption that $d(x_{kn+i-1},x_{kn+i})\leq \lambda_k L$,
and hence
\[ f_i(x_{kn})-f_i(x_{kn+i-1})\leq
 Ld(x_{kn},x_{kn+i-1})\leq L\sum_{j=1}^{i-1}d(x_{kn+j-1},x_{kn+j})\leq \lambda_k L^2(i-1). \]
Then these bounds combined with \eqref{eq:T:1.4} give
\[ d(y,x_{kn+n})^2\leq
 d(y,x_{kn})^2-2\lambda_k\sum_{i=1}^n[f_i(x_{kn})-f_i(y)]+\lambda_k^2L^2n\left(\frac{K}{2}+n-1\right). \]
The rest of the argument is identical to Case (I).
\end{proof}

\begin{remark}
In the lower curvature bound case (II), the assumption that $x_m \in G$ can be met,
since $d(y,x_m)$ is bounded as we saw in the proof.
Thus, choosing the sequence $\lambda_k$ appropriately, we can assure that $x_m$ stays inside $G$.
\end{remark}

The above theorem relies on local compactness.
In fact, it is known that in the infinite dimensional case we cannot always have convergence under these assumptions \cite{Ba2}.
However, if we assume that $f$ is $K$-convex for positive $K$, then the assumption of local compactness can be dropped.

\begin{proposition}\label{P:1}
Let $(X,d)$, $G \subset X$, $f=\sum_{i=1}^n f_i$ be as in Definition~$\ref{D:alg1}$
and further assume that $f_i$ is $L$-Lipschitz for some $L \ge 1$ and all $i$,
and that $f$ is $K$-convex for some $K>0$.
Take $\lambda_k>0$ with $\lambda_kK<1$, $\lambda_k\to 0$ and $\sum_{k=0}^\infty\lambda_k=+\infty$,
and consider a sequence $\{x_m\}_{m \ge 0}$ generated by Definition~$\ref{D:alg1}$.
Then $x_m$ converges to the unique minimizer $y \in G$ of $f$ as $m \to \infty$.

More concretely, in the upper curvature bound case, $d(x_{kn},y)^2\leq a_k$ holds with
$a_0:=d(x_0,y)^2$ and $a_{k+1}:=(1-\lambda_kK)a_k+2\lambda_k^2 L^2n(n+1)$ inductively, that is,
\[ a_{k+1}=\prod_{i=0}^k(1-\lambda_iK)a_0
 +2L^2n(n+1)\sum_{j=1}^k(\lambda_{j-1}^2\prod_{i=j}^k(1-\lambda_iK)+\lambda_k^2). \]
In the lower curvature bound case, $d(x_{kn},y)^2\leq a_k$ similarly holds for
$a_0:=d(x_0,y)^2$ and $a_{k+1}:=(1-\lambda_kK)a_k+\lambda_k^2 L^2n(\frac{K}{2}+n-1)$
with $K \ge 0$ given in Lemma~$\ref{L:1}${\rm (II)}.
Also
\[ a_{k+1}=\prod_{i=0}^k(1-\lambda_iK)a_0
 +L^2n \left(\frac{K}{2}+n-1 \right) \sum_{j=1}^k(\lambda_{j-1}^2\prod_{i=j}^k(1-\lambda_iK)+\lambda_k^2) \]
in this case.
\end{proposition}

\begin{proof}
Thanks to the $K$-convexity with $K>0$ and the completeness of $(X,d)$,
there is a unique minimizer $y \in G$ of $f$ (see, e.g., \cite[Lemma~2.4.8]{ags}).
For any $x \in G$, by dividing \eqref{k-convexity} with $1-t$ and letting $t \to 1$, we have
\begin{equation}\label{eq:P:1}
\frac{K}{2}d(x,y)^2\leq f(x)-f(y).
\end{equation}

Let us consider Case (I), the proof of Case (II) will be similar.
By \eqref{eq:T:1.3}, we have
\begin{equation*}
d(y,x_{kn+n})^2\leq d(y,x_{kn})^2-2\lambda_k[f(x_{kn})-f(y)]+2\lambda_k^2 L^2n(n+1).
\end{equation*}
Using \eqref{eq:P:1}, we get
\begin{equation}\label{eq:C:2}
d(y,x_{kn+n})^2\leq (1-\lambda_kK)d(y,x_{kn})^2+2\lambda_k^2 L^2n(n+1).
\end{equation}
Then by induction it is easy to see that $d(x_{kn},y)^2\leq a_k$.
The explicit formula for $a_{k+1}$ is proved also by induction.

Now we prove $\liminf_{k\to\infty}a_k=0$ by contradiction.
Assume that there are $N \ge 0$ and $c>0$ such that, for every $k>N$,
we have $a_k>c$ and $2L^2n(n+1)\lambda_k <cK/2$.
Then
\[ a_{k+1}=a_k+\lambda_k(2L^2n(n+1)\lambda_k-a_kK)\leq a_k-\frac{\lambda_kcK}{2}, \]
which is a contradiction, since $\sum_{k=0}^{\infty}\lambda_k=+\infty$.
We finally show $\lim_{k \to \infty}a_k=0$.
If $a_k>2L^2n(n+1)\lambda_k/K$, then clearly $a_{k+1}<a_k$.
If $a_k\leq 2L^2n(n+1)\lambda_k/K$, then
\[ a_{k+1} \le (1-\lambda_kK)(2L^2n(n+1)\lambda_k/K)+2\lambda_k^2 L^2n(n+1)=2L^2n(n+1)\lambda_k/K. \]
Thus we have
\[ a_{k+1}\leq \max\{a_k,2L^2n(n+1)\lambda_k/K\}, \]
from which we get, for any $l \ge k$,
\[ a_{l+1}\leq \max\left\{a_k,(2L^2n(n+1)/K) \cdot \max\{\lambda_k,\lambda_{k+1},\ldots,\lambda_l\}\right\}. \]
Take $\limsup_{l\to\infty}$ and then $\liminf_{k\to\infty}$ of the above to see that $a_k\to 0$. The convergence of the rest of the sequence $d(x_m,y)^2$ to $0$ follows from setting up a similar inequality of the form \eqref{eq:T:1.2}.
\end{proof}

For the explicit convergence rate analysis, let us quote a lemma from \cite{Ne1}:
\begin{lemma}\label{L:convrate}
Let $a_k\geq 0$ be a sequence such that
\[ a_{k+1}\leq \left(1-\frac{\alpha}{k+1}\right)a_k+\frac{\beta}{(k+1)^2}, \]
where $\alpha,\beta>0$. Then
\[ a_{k}\leq
\begin{cases}
\frac{1}{(k+2)^{\alpha}}\left(a_0+\frac{2^\alpha\beta(2-\alpha)}{1-\alpha}\right) & \mbox{if }0<\alpha<1;\\
\frac{\beta(1+\log(k+1))}{k+1} & \mbox{if }\alpha=1;\\
\frac{1}{(\alpha-1)(k+2)}\left(\beta+\frac{(\alpha-1)a_0-\beta}{(k+2)^{\alpha-1}}\right) & \mbox{if }\alpha>1.
\end{cases} \]
\end{lemma}
From this we obtain that the convergence is sublinear in Proposition~\ref{P:1}.

\section{Law of large numbers and Jensen's inequality}\label{S:LLN}

In this section, we give a stochastic discrete-time gradient flow for arbitrary convex (infinite) combinations of convex functions.
We will restrict ourselves to $K$-convex functions with $K>0$,
however, our proofs can be adapted to the case $K=0$ in the same manner as we have seen in Theorem~\ref{T:1},
which is a generalized form of Proposition~\ref{P:1} in this sense. We will adopt, and generalize the notations of \cite{St02b,St03} given for measures supported only over the squared distance functions $f_a(x):=d(x,a)^2$.

Let $G \subset X$ be a closed, geodesically convex set. We assume that $(G,d)$ is separable. Consider the set of all lower semi-continuous, convex functions $f:G\to(-\infty,\infty]$ not identically $+\infty$, denoted by $F(G)$. For $K>0$, we denote by $F_K(G)$ the subset of all lower semi-continuous, $K$-convex functions $f:G\to(-\infty,\infty]$ not identically $+\infty$. In order to consider measures over $F_K(G)$, we must equip $F(G)$ with a $\sigma$-algebra. There are different ways to do this, however there is a natural topology on $F(G)$ that is obtained by associating every function $f\in F(G)$ with its epigraph $\mathrm{epi}(f):=\{(x,\alpha)\in G\times(-\infty,+\infty]:\alpha\geq f(x)\}$. It is known that $f$ is convex lower semi-continuous if and only if $\mathrm{epi}(f)$ is a closed convex set of $(G\times(-\infty,+\infty])$ which itself is equipped with the product topology. The construction of the topology we adopt is a standard one in stochastic variational analysis, we refer to the book \cite{RockWets98}, where instead of an arbitrary Polish space $X$, only finite dimensional Euclidean spaces are considered, however the theory carries over without modifications to the general case, as can be seen in \cite{KorfWets01} for example.

The set of closed convex sets of $(G\times(-\infty,+\infty])$ is denoted by $\mathrm{clc}(G\times(-\infty,+\infty])$. The \emph{Effr\"os-field} on $\mathrm{clc}(G\times(-\infty,+\infty])$ is the $\sigma$-field $\mathcal{E}(G\times(-\infty,+\infty])$ generated by all sets of the form
\begin{equation*}
\mathcal{E}_{O}=\{C\in\mathrm{clc}(G\times(-\infty,+\infty]):C\cap O\neq\emptyset\},\quad O\subseteq G\times(-\infty,+\infty],\quad\text{open}.
\end{equation*}
The topology on $F(G)$ is then generated by the topology on $\mathrm{clc}(G\times(-\infty,+\infty])$ given by the $\sigma$-field $\mathcal{E}(G\times(-\infty,+\infty])$ which is itself generated by the \emph{Fell} or \emph{Choquet-Wijsman} hyperspace topologies, see \cite{KorfWets01} and the references therein. The resulting $\sigma$-field on $F(G)$ is denoted by $\mathcal{E}$. It is known that $\mathcal{E}$ is generated by sets of the form
\begin{equation*}
\mathcal{E}_{(O,\alpha)}=\left\{f\in F(G):\inf_Of<\alpha\right\},\quad O\subseteq G,\quad\text{open}, \quad \alpha\in\mathbb{R},
\end{equation*}
see \cite{KorfWets01}, it corresponds to a topology of one-sided uniform convergence. Now we can consider measures on $(F_K(G),\mathcal{E})$, i.e. random lower semi-continuous $K$-convex functions on $G$. Let $(\Pi,\mathcal{A},\mu)$ be a complete probability space. Then a map $L:\Pi\to F_K(G)$ is a \emph{random lower semi-continuous ($K$-)convex function} if the bivariate map $(\Pi,G)\to L(a,x)$ is $\mathcal{A}\otimes\mathcal{B}(X)$-measurable where $\mathcal{B}(X)$ denotes the Borel $\sigma$-algebra of $(G,d)$. Equivalently $L:\Pi\to F_K(G)$ is a \emph{random lower semi-continuous ($K$-)convex function} if the associated epigraphical mapping $S_L(a):=\mathrm{epi}\,L(a,\cdot)=\{(x,\alpha)\in G\times\mathbb{R}:\alpha\geq L(a,x)\}$ is $\mathcal{E}(G\times(-\infty,+\infty])$ measurable as a closed convex set valued mapping, see Proposition~14.34 in \cite{RockWets98}.

A very useful consequence of the measurability of $L:\Pi\to F_K(G)$ is the following:
\begin{lemma}\label{L:3}
The resolvent map $\Pi\mapsto J_{\lambda}^{L(a)}(x)$ defined by \eqref{eq:resolventabove} for fixed $x\in G$, as a map $\Pi\to X$, is closed-valued and $\mathcal{A}$-measurable as a set-valued map, moreover $\Pi\mapsto L(a,J_{\lambda}^{L(a)}(x))$, as a map $\Pi\to \mathbb{R}$ is also $\mathcal{A}$-measurable.
\end{lemma}
\begin{proof}
See Theorem 14.37 in \cite{RockWets98}.
\end{proof}
Also by the measurable projection theorem we have that for fixed $x\in G$ the map $L(\cdot,x):\Pi\to (-\infty,+\infty]$ is $\mathcal{A}$-measurable, see Proposition~14.28 in \cite{RockWets98}. What follows is that the integral $f(x)=\int_{\Pi}L(a,x)d\mu(a)$ pointwisely defines an extended real-valued function $f:G\to[-\infty,+\infty]$.

\begin{lemma}\label{L:4}
The function $f:G\to[-\infty,+\infty]$ defined as $f(x):=\int_{\Pi}L(a,x)d\mu(a)$ is lower semi-continuous $K$-convex and $f(x)>-\infty$ for all $x\in G$ if there exists an integrable function $\alpha_0:\Pi\to(-\infty,+\infty)$ such that $L(a,x)\geq \alpha_0(a)$ holds almost surely.
\end{lemma}
\begin{proof}
The lower boundedness of $f$ is clear under the last part of the assumption.

To prove the first part, let $x_k\in G$ be a sequence such that $x_k\to x$. We have $f(x)=\int_{\Pi}L(a,x)d\mu(a)$ and $f(x_k)=\int_{\Pi}L(a,x_k)d\mu(a)$. For fixed $a\in\Pi$ by lower semi-continuity we have $L'(a,x):=\lim_{x_k\to x}L(a,x_k)\geq L(a,x)$. Hence, by monotonicity of the Lebesgue-integral we get $\int_{\Pi}L'(a,x)d\mu(a)\geq \int_{\Pi}L(a,x)d\mu(a)$. Then by Fatou's lemma we get $\int_{\Pi}L'(a,x)d\mu(a)\leq \liminf_{k\to\infty}\int_{\Pi}L(a,x_k)d\mu(a)$, hence $f(x)\leq \liminf_{k\to\infty}f(x_k)$ proving the lower semi-continuity. Now the $K$-convexity of $f$ is obtained by integrating the inequality \eqref{k-convexity} given for $x\mapsto L(a,x)$ for fixed $a\in\Pi$.
\end{proof}

With the above setup at hand, instead of always emphasizing the complete probability space $(\Pi,\mathcal{A},\mu)$, we assume directly that $\Pi:=F_K(G)$, $\mathcal{A}:=\mathcal{E}$ and $\mu$ is a complete probability measure on $(\Pi,\mathcal{A})$. By the definition of a random lower semi-continuous $K$-convex function it follows that the map $(\Pi,G)\to L(a,x)$ is $\mathcal{A}\otimes\mathcal{B}(X)$-measurable, see Exercise~14.9 in \cite{RockWets98}, hence the above machinery applies. For simplicity we denote by $\mathfrak{P}(F_K(G))$ the set of all complete probability measures on $F_K(G)$ with $\sigma$-field $\mathcal{A}=\mathcal{E}$, such that $g(x):=\int_{F_K(G)}f(x)d\mu(f)$ is lower semi-continuous $(-\infty,+\infty]$-valued $K$-convex and there exists $x\in G$ so that $g(x)<+\infty$.

\begin{definition}[Variance]
We define the \emph{variance} of $\mu\in \mathfrak{P}(F_K(G))$ by
\[ \var(\mu):=\inf_{x\in G}\int_{F_K(G)}f(x)d\mu(f). \]
\end{definition}
This contains as a special case the original definition of the variance given by $\var(\nu):=\inf_{x\in G}\int_{G}d(x,a)^2d\nu(a)$ in \cite{St02b,St03} for a probability measure $\nu$ supported over $G$.

A fixed $\mu\in \mathfrak{P}(F_K(G))$ can be viewed as the distribution of an $F_K(G)$-valued random variable.
In this sense, integration with respect to $\mu$ can be viewed as taking expectations:
\[ \mathbb{E}\varphi:=\int_{F_K(G)}\varphi(f)d\mu(f), \]
where $\varphi:F_K(G)\to[-\infty,+\infty]$ is assumed to be measurable.

\begin{definition}[Expectation]
Let $\mu\in \mathfrak{P}(F_K(G))$.
We define the \emph{expectation} of $\mu$ as
\[ \mathbb{E}\mu:=\argmin_{x\in G}\int_{F_K(G)}f(x)d\mu(f), \]
which is indeed uniquely determined by the $K$-convexity of $g(x)=\int_{F_K(G)}f(x)d\mu(f)$.
\end{definition}

The above is motivated by the definition given in \cite{St02b,St03} of the expectation as
$\mathbb{E}\nu:=\argmin_{x\in G}\int_{G}d(x,a)^2d\nu(a)$ of a probability measure $\nu$ supported over $G$.

Note that $g(\mathbb{E}\mu)=\var(\mu)$.
Using our new notation, we have a generalization of the \emph{variance inequality} in \cite[Proposition~4.4]{St03}
as well (see also \cite{St05,Oh3} for the \emph{reverse} variance inequality
for squared distance functions under lower curvature bounds).
Let $L_x$ denote the evaluation operator at $x\in G$ defined as $L_xf:=f(x)$.
Clearly $L_x$ is a linear functional on the cone $F_K(G)$.

\begin{proposition}[Variance inequality]
Let $\mu\in \mathfrak{P}(F_K(G))$.
Then, for all $x\in G$, we have
\begin{equation}\label{eq:varineq}
d(x,\mathbb{E}\mu)^2\leq
 \frac{2}{K}\mathbb{E}\left(L_x-L_{\mathbb{E}\mu}\right)=\frac{2}{K}\int_{F_K(G)}[f(x)-f(\mathbb{E}\mu)]d\mu(f).
\end{equation}
\end{proposition}

\begin{proof}
Put $g(x)=\int_{F_K(G)}f(x)d\mu(f)$ and note that $g(\mathbb{E}\mu)=\inf_G g$.
Then the claim follows from \eqref{eq:P:1}.
\end{proof}

\begin{remark}\label{R:1}
Lemma~\ref{L:3} ensures us that, in the case of upper curvature bound,
the nonnegative real-valued map $f\mapsto d(y,J_{\lambda}^f(x))^2$ is measurable for any $x,y\in G$, i.e.
\begin{equation}\label{eq:int}
\int_{F_K(G)}d(y,J_{\lambda}^f(x))^2d\mu(f)
\end{equation}
exists. 
In the case of lower curvature bound, the measurability of $f\mapsto d(y,J_{\lambda}^f(x))^2$ is nontrivial
and verified only in special cases.
If $\mu$ is finitely supported, then measurability is clear.
Also if $X$ is a Euclidean space and $\mu$ is supported over differentiable functions,
then the measurability follows from the continuity of the gradient vectors of convex functions,
see Theorem~25.7 in \cite{Rock97}.
\end{remark}

In the following, we prove a stochastic variant of Proposition~\ref{P:1},
which extends the law of large numbers proved in \cite[Theorem~4.7]{St03} to the case of Alexandrov spaces
with arbitrary upper or lower curvature bounds, and arbitrary Lipschitz functions in $F_K(G)$.

\begin{theorem}[Law of large numbers]\label{T:2}
Let $(X,d)$ and $G \subset X$ be as in Definition~$\ref{D:alg1}$.
Fix $\mu\in \mathfrak{P}(F_K(G))$ supported on $L$-Lipschitz functions and let
$\{f_k\}_{k\geq 0}$ denote a sequence of independent, identically distributed random variables
taking values in $F_K(G)$ with distribution $\mu$.
Take a positive sequence $\{\lambda_k\}_{k\geq 0}$ with $\lambda_kK<1$, $\lambda_k\to 0$
and $\sum_{k=0}^\infty\lambda_k=+\infty$.
Define the sequence $S_k \in G$ recursively as
\[ S_{k+1}:=J_{\lambda_k}^{f_k}(S_k), \quad k \ge 0, \]
with an arbitrary starting point $S_0\in G$,
assuming that $S_k \in G$ for all $k \ge 0$ and the integral in \eqref{eq:int} exists
in the lower curvature bound case.
Then $S_k\to\mathbb{E}\mu$ almost surely.
\end{theorem}

\begin{proof}
We prove only the upper curvature bound case, the lower curvature bound case is similar.
By \eqref{eq:L:1.1} in Lemma~\ref{L:1}, we have
\[ d(y,J_{\lambda_k}^{f_k}(x))^2\leq d(y,x)^2-2\lambda_k[f_k(J_{\lambda_k}^{f_k}(x))-f_k(y)] \]
for all $x,y \in G$.
Therefore we have
\begin{equation}\label{eq:T:2.1}
d(\mathbb{E}\mu,S_{k+1})^2\leq
 d(\mathbb{E}\mu,S_k)^2-2\lambda_k[f_k(S_k)-f_k(\mathbb{E}\mu)]+2\lambda_k[f_k(S_k)-f_k(S_{k+1})].
\end{equation}
By \eqref{eq:resolventabove}, we have
\[ f_k(S_{k+1})+\frac{1}{2\lambda_k}d(S_{k+1},S_{k})^2\leq f_k(S_{k}), \]
which yields by using the $L$-Lipschitz continuity that
\[ d(S_{k+1},S_{k})\leq 2\lambda_k\frac{f_k(S_{k})-f_k(S_{k+1})}{d(S_{k+1},S_{k})}\leq 2\lambda_k L. \]
Thus we obtain
\[ f_k(S_{k})-f_k(S_{k+1})\leq Ld(S_{k+1},S_{k})\leq 2\lambda_k L^2. \]
This combined with \eqref{eq:T:2.1} yields
\[ d(\mathbb{E}\mu,S_{k+1})^2\leq
 d(\mathbb{E}\mu,S_k)^2-2\lambda_k[f_k(S_k)-f_k(\mathbb{E}\mu)]+4\lambda_k^2L^2. \]
Taking expectations in $f_k$ conditioned on $\mathcal{F}_{k-1}:=\{f_1,\ldots,f_{k-1}\}$ and using the variance inequality \eqref{eq:varineq}, we get
\[ \begin{split}
\mathbb{E}\left(d(\mathbb{E}\mu,S_{k+1})^2|\mathcal{F}_{k-1}\right)
 &\leq d(\mathbb{E}\mu,S_k)^2-2\lambda_k\mathbb{E}[f_k(S_k)-f_k(\mathbb{E}\mu)]+4\lambda_k^2L^2\\
 &\leq d(\mathbb{E}\mu,S_k)^2-\lambda_k Kd(\mathbb{E}\mu,S_k)^2+4\lambda_k^2L^2,
\end{split} \]
and hence
\[ \mathbb{E}\left(d(\mathbb{E}\mu,S_{k+1})^2|\mathcal{F}_{k-1}\right)\leq
 (1-\lambda_kK)d(\mathbb{E}\mu,S_k)^2+4\lambda_k^2L^2. \]
Taking expectations again yields
\[ \mathbb{E}d(\mathbb{E}\mu,S_{k+1})^2\leq
 (1-\lambda_kK)\mathbb{E}d(\mathbb{E}\mu,S_k)^2+4\lambda_k^2L^2. \]
From here proving the convergence $\mathbb{E}d(\mathbb{E}\mu,S_{k+1})^2\to 0$ can be done in the same way
as in the proof of Proposition~\ref{P:1} after \eqref{eq:C:2}.
To get a convergence rate estimate, one can refer to Lemma~\ref{L:convrate}.
\end{proof}

\begin{remark}\label{R:2}
Suppose that $(X,d)$ has curvature bounded above by $\kappa>0$.
Fix arbitrary $o\in X$ and let $G:=\bar{B}_o(r)$ with $2r=(\pi/2-\epsilon)/\sqrt{\kappa}$
for $\epsilon\in(0,\pi/2)$.
Then, by Proposition~\ref{P:Convex}, the function $f_a(x):=d(a,x)^2$ with $a\in G$ is $K$-convex and Lipschitz continuous
on $G$ with $K=(\pi-2\epsilon)\tan\epsilon>0$.
Take $\mu\in \mathfrak{P}(F_K(G))$ such that $\supp\mu\subset\{f_a:a\in G\}$.
Then Theorem~\ref{T:2} generalizes Sturm's law of large numbers in \cite{St02b,St03}.
In particular, if $\lambda_k:=\frac{1}{2k}$, then we have
\begin{align*}
S_{k+1}&=J_{\lambda_k}^{f_k}(S_k)
 =\argmin_{z\in G} \left\{ d(a_k,z)^2+\frac{d(z,S_k)^2}{2\lambda_k} \right\}\\
&=\argmin_{z\in G} \left\{ \frac{2\lambda_k}{1+2\lambda_k}d(a_k,z)^2+\frac{1}{1+2\lambda_k}d(z,S_k)^2 \right\}\\
&=S_k\#_{\frac{1}{k+1}}a_k,
\end{align*}
where $a_k$ is a $G$-valued random variable with distribution provided by the push-forward measure of $\mu$
under the bijective map $f_a \mapsto a$.
In this case, one can reproduce the same sublinear order of convergence $O(1/k)$ as in \cite{St03}.
More generally, one can consider $f_a(x):=d(x,a)^p$ for any $p\in [2,\infty)$,
still $f_a(x)$ being $K$-convex and Lipschitz continuous on the same $G$,
hence Theorem~\ref{T:2} can be applied.

Under these assumptions if $\mu$ is also finitely supported, then our Proposition~\ref{P:1} extends the ``no dice'' approximation given only for the barycenter on NPC spaces in \cite{LP2,Hol}.
\end{remark}

It seems reasonable to expect that, in the upper curvature bound case in Remark~\ref{R:2},
one can take any $2r<\pi/\sqrt{\kappa}$ even though the functions $f_a$ are then not convex on whole $G$.
This is motivated by the results in \cite{afs} on the existence and uniqueness of the center of mass in Riemannian manifolds.

\begin{remark}\label{R:3}
Theorem~\ref{T:2} generalizes the law of large numbers from Euclidean spaces to Alexandrov spaces,
moreover, to the case of measures supported over the cone of $K$-convex Lipschitz functions.
We recover the original law of large numbers in Hilbert spaces by choosing $X$ to be a Hilbert space in Remark~\ref{R:2}.
Also the setting in Remark~\ref{R:2} is of interest if we choose $X$ to be a sphere,
or any compact Lie group with a bi-invariant Riemannian metric, for instance the matrix Lie group of rotations $\mathrm{SO}(\mathcal{H})$ studied in \cite{Moakher} or unitary tranformations $\mathrm{U}(\mathcal{H})$ over a finite dimensional Hilbert space $\mathcal{H}$.
\end{remark}

Using our law of large numbers, we have an alternative proof for Jensen's inequality of Kuwae~\cite{Kuw},
along the line of the second proof of \cite[Theorem~6.2]{St03} in the NPC space case.

\begin{proposition}[Jensen's inequality]\label{P:jensen}
Let $X$ be a complete Alexandrov space with curvature bound above by $\kappa>0$,
and $G\subset X$ be a closed, geodesically convex set with $\diam G<\pi/(2\sqrt{\kappa})$.
Take a probability measure $\mu$ on $G$ and a convex, lower semi-continuous function $f:G \to \mathbb{\R}$.
Then we have
\[ f(\mathbb{E}\mu)\leq \mathbb{E}f, \]
where $\mathbb{E}\mu:=\argmin_{y\in G}\int_{G}d(a,y)^2 d\mu(a)$ and $\mathbb{E}f:=\int_{G}f(a)d\mu(a)$.
\end{proposition}

\begin{proof}
Choose a sequence of independent, identically distributed random variables $Y_k$ with values in $G$,
and with distribution $\mu$.
Let $S_k \in G$ be defined as in Remark~\ref{R:2}, i.e.,
$S_1:=Y_1$ and $S_{k+1}:=S_k\#_{\frac{1}{k+1}}Y_{k+1}$.
Similarly, let $Z_k \in \mathbb{R}$ be defined as $Z_1:=f(Y_1)$ and $Z_{k+1}:=Z_k\#_{\frac{1}{k+1}}f(Y_{k+1})$.
We can explicitly write as
\[ Z_{k+1}=Z_k\#_{\frac{1}{k+1}}f(Y_{k+1})=\frac{k}{k+1}Z_k+\frac{1}{k+1}f(Y_{k+1}). \]
By Theorem~\ref{T:2}, we have $S_k\to\mathbb{E}\mu$ and $Z_k\to\mathbb{E}(f_* \mu)=\mathbb{E}f$,
where $f_* \mu$ denoted the push-forward of $\mu$.
We proceed by induction showing
\begin{equation}\label{eq:jensen.1}
f(S_k)\leq Z_k.
\end{equation}
For $k=1$, this clearly holds.
For general $k \ge 1$, we have by induction
\begin{align*}
f(S_{k+1}) &=f(S_k\#_{\frac{1}{k+1}}Y_{k+1})\\
&\leq \frac{k}{k+1}f(S_k)+\frac{1}{k+1}f(Y_{k+1})\\
&\leq \frac{k}{k+1}Z_k+\frac{1}{k+1}f(Y_{k+1})=Z_{k+1}
\end{align*}
showing \eqref{eq:jensen.1}.
Hence, by the lower semi-continuity of $f$, we obtain
\[ f(\mathbb{E}\mu)\leq \liminf_{k\to\infty}f(S_k)\leq \liminf_{k\to\infty}Z_k=\mathbb{E}f \]
and complete the proof.
\end{proof}

\section*{Acknowledgment}
The authors would like to thank the anonymous referee for his valuable comments, in particular improving the discussion in section 6.

The second author would like to thank Prof.~John Holbrook for raising his attention to the approximation problem
of the barycenter treated in Remark~\ref{R:2} on the sphere.
The second author had doubts in the convergence of such approximation scheme in the positive curvature case,
but then he learned about the favorable outcomes of Prof.~Holbrook's numerical experiments on the sphere
in a private communication with him, which initiated the further study of the problem.


\begin{thebibliography}{99}

\bibitem{afs}
B. Afsari, Riemannian $L^{p}$ center of mass: Existence, uniqueness, and convexity,
Proc. Amer. Math. Soc. {\bf 139} (2011), 655--673.

\bibitem{ags}
L. Ambrosio, N. Gigli and G. Savar\'e, Gradient flows in metric spaces and in the space of probability measures. Second edition,
Birkh\"auser Verlag, Basel, 2008.

\bibitem{Arn05}
M. Arnaudon and X. M. Li, Barycenters of measures transported by stochastic flows,
Ann. Probab. {\bf 33} (2005), 1509--1543.

\bibitem{Arn10}
M. Arnaudon, C. Dombry, A. Phan and L. Yang,
Stochastic algorithms for computing means of probability measures,
Stochastic Process. Appl. {\bf 122} (2012), 1437--1455.

\bibitem{Ba1}
M. Ba\v{c}\'ak, The proximal point algorithm in metric spaces,
Israel J. Math. {\bf 194} (2013), 689--701.

\bibitem{Ba2}
M. Ba\v{c}\'ak, Computing means and medians in Hadamard spaces,
to appear in SIAM J. Optim. (2014). Available at {\sf arXiv:1210.2145}.

\bibitem{Ben12}
G. C. Bento and J. X. Cruz Neto,
Finite termination of the proximal point method for convex functions on Hadamard manifolds,
Optimization (2012), DOI:10.1080/02331934.2012.730050.

\bibitem{Bert2}
D. P. Bertsekas, Incremental proximal methods for large scale convex optimization,
Math. Program., Ser. B {\bf 129} (2011), 163--195.

\bibitem{Bert1}
D. P. Bertsekas and J. N. Tsitsiklis, Neuro-Dynamic Programming, Athena Scientific, 1996.

\bibitem{Bha}
R. Bhatia, Positive definite matrices,
Princeton Series in Applied Mathematics, Princeton University Press, Princeton, NJ, 2007.

\bibitem{BH}
R. Bhatia and J. Holbrook, Riemannian geometry and matrix geometric means,
Linear Algebra Appl. \textbf{413} (2006), 594--618.

\bibitem{Brezis}
H. Br\'ezis and P.-L. Lions, Produits infinis de r\`esolvantes,
Israel J. Math. {\bf 29} (1978), 329--345.

\bibitem{bbi}
D. Burago, Yu. Burago and S. Ivanov, A course in metric geometry, American Mathematical Society, Providence, RI, 2001.

\bibitem{Esp09}
R. Esp\'inola and A. Fern\'andez-Le\'on, CAT$(k)$-spaces, weak convergence and fixed points,
J. Math. Anal. Appl. {\bf 353} (2009), 410--427.

\bibitem{Fer02}
O. P. Ferreira and P. R. Oliveira, Proximal Point Algorithm On Riemannian Manifolds,
Optimization {\bf 51} (2002), 257--270.

\bibitem{Hol}
J. Holbrook, No dice: a deterministic approach to the Cartan centroid,
J. Ramanujan Math. Soc. {\bf 27} (2012), 509--521.

\bibitem{Jost-Calc}
J.~Jost, Equilibrium maps between metric spaces,
Calc.\ Var.\ Partial Differential Equations {\bf 2} (1994), 173--204.

\bibitem{Jost}
J.~Jost, Convex functionals and generalized harmonic maps into spaces of nonpositive curvature,
Comment. Math. Helv. \textbf{70} (1995), 659--673.

\bibitem{Jost-NPC}
J.~Jost, Nonpositive curvature: geometric and analytic aspects,
Birkh\"auser Verlag, Basel, 1997.

\bibitem{Jost-NLD}
J.~Jost, Nonlinear Dirichlet forms, New directions in Dirichlet forms, 1--47, 
AMS/IP Stud.\ Adv.\ Math., {\bf 8}, Amer.\ Math.\ Soc., Providence, RI, 1998.

\bibitem{Ka77}
H. Karcher, Riemannian center of mass and mollifier smoothing,
Comm. Pure Appl. Math. \textbf{30} (1977), 509--541.

\bibitem{Ken90}
W. S. Kendall,
Probability, convexity, and harmonic maps with small image I: uniqueness and fine existence,
Proc. London Math. Soc. (3) {\bf 61} (1990), 371--406.

\bibitem{Ken91}
W. S. Kendall, Convexity and the hemisphere, J. London Math. Soc. (2) {\bf 43} (1991), 567--576.

\bibitem{KorfWets01}
L. A. Korf and R. J.-B. Wets, Random lsc Functions: An Ergodic Theorem,  Mathematics of Operations Research (26) {\bf 2} (2001), 421--445.

\bibitem{Kuw}
K.~Kuwae, Jensen's inequality over CAT$(\kappa)$-space with small diameter,
Potential theory and stochastics in Albac, 173--182,
Theta Ser.\ Adv.\ Math., {\bf 11}, Theta, Bucharest, 2009.

\bibitem{LL1}
J. Lawson and Y.  Lim, Monotonic properties of the least squares mean,
Math. Ann. \textbf{351} (2011), 267--279.

\bibitem{Li}
C. Li, G. L\'opez, and V. Mart\'in-M\'arquez,
Monotone vector fields and the proximal point algorithm on Hadamard manifolds,
J. Lond. Math. Soc. (2) {\bf 79} (2009), 663--683.

\bibitem{LP2}
Y. Lim and M.~P\'alfia, Weighted deterministic walks for the least squares mean on Hadamard spaces,
To appear in Bull. London Math. Soc.

\bibitem{Ly}
A. Lytchak, Open map theorem for metric spaces,
St. Petersburg Math. J. \textbf{17} (2006), 477--491.

\bibitem{Mayer}
U. F. Mayer, Gradient flows on nonpositively curved metric spaces and harmonic maps,
Comm. Anal. Geom. {\bf 6} (1998), 199--253.

\bibitem{Moakher}
M. Moakher, Means and Averaging in the Group of Rotations,
SIAM J. Matrix Anal. Appl. {\bf 24} (2002), 1--16.

\bibitem{Ne1}
A. Nedic and D. P. Bertsekas,
Convergence Rate of Incremental Subgradient Algorithms,
Stochastic optimization: algorithms and applications (Gainesville, FL, 2000), 223--264, 
Appl. Optim., {\bf 54}, Kluwer Acad. Publ., Dordrecht, 2001.

\bibitem{Ne2}
A. Nedic and D. P. Bertsekas, Incremental subgradient methods for nondifferentiable optimization,
SIAM J. Optim. {\bf 12} (2001), 109--138.

\bibitem{Oh2}
S. Ohta, Convexities of metric spaces, Geom.\ Dedicata \textbf{125} (2007), 225--250.

\bibitem{Oh1}
S. Ohta, Gradient flows on Wasserstein spaces over compact Alexandrov spaces,
Amer.\ J.\ Math.\ \textbf{131} (2009), 475--516.

\bibitem{Oh3}
S.~Ohta, Barycenters in Alexandrov spaces of curvature bounded below,
Adv.\ Geom.\ {\bf 12} (2012), 571--587.

\bibitem{PP}
G.~Perel'man and A.~Petrunin, Quasigeodesics and gradient curves in Alexandrov spaces,
Unpublished preprint (1995).
Available at {\sf http://www.math.psu.edu/petrunin/}

\bibitem{Pet07}
A.~Petrunin, Semiconcave functions in Alexandrov's geometry,
Surveys in differential geometry.\ Vol. XI, 137--201, 
Surv.\ Differ.\ Geom., {\bf 11}, Int.\ Press, Somerville, MA, 2007.

\bibitem{Rock97}
R.T.~Rockafellar, Convex Analysis, Princeton University Press, 1997.

\bibitem{RockWets98}
R.T.~Rockafellar and R.~J.-B.~Wets, Variational Analysis, Springer-Verlag, Berlin, 1998.

\bibitem{St01}
K.-T.~Sturm, Nonlinear Markov operators associated with symmetric Markov kernels
and energy minimizing maps between singular spaces,
Calc. Var. Partial Differential Equations {\bf 12} (2001), 317--357.

\bibitem{St02a}
K.-T.~Sturm, Nonlinear Markov operators, discrete heat flow, and harmonic maps between singular spaces,
Potential Anal. {\bf 16} (2002), 305--340.

\bibitem{St02b}
K.-T.~Sturm, Nonlinear martingale theory for processes with values in metric spaces of nonpositive curvature,
Ann. Probab. {\bf 30} (2002), 1195--1222.

\bibitem{St03}
K.-T.~Sturm, Probability measures on metric spaces of nonpositive curvature,
Heat kernels and analysis on manifolds, graphs, and metric spaces (Paris, 2002), 357--390, 
Contemp. Math., {\bf 338}, Amer. Math. Soc., Providence, RI, 2003.

\bibitem{St05}
K.-T.~Sturm, A semigroup approach to harmonic maps,
Potential Anal.\ {\bf 23} (2005), 225--277.

\bibitem{St061}
K.-T.~Sturm, On the geometry of metric measure spaces,
Acta Math. \textbf{196} (2006), 65--131.

\end{thebibliography}
\end{document}